\documentclass[12pt]{amsart}
\usepackage{amscd,amsmath,amsthm,amssymb,graphics}

\usepackage{amsmath}
\usepackage{amssymb}
\usepackage{amsbsy}
\usepackage{graphicx}
\usepackage{amsfonts}

\usepackage{amscd,amsmath,amsthm,amssymb}
\usepackage{pstcol,pst-plot,pst-3d}%\usepackage[T1]{fontenc}
\usepackage{lmodern,pst-node}%\usepackage{geometry}

\usepackage{pgf,tikz}
\usetikzlibrary{arrows}

\def\frk{\frak}               % font for "Fraktur"

\def\Phi{{\frk n}}
\def\Phi{{\frk N}}
\def\opn#1#2{\def#1{\operatorname{#2}}} % to make operators
\opn\chara{char} \opn\length{\ell} \opn\pd{pd} \opn\rk{rk}
\opn\projdim{proj\,dim} \opn\injdim{inj\,dim} \opn\rank{rank}
\opn\depth{depth} \opn\grade{grade} \opn\height{height}
\opn\embdim{emb\,dim} \opn\codim{codim}

\opn\Tr{Tr} \opn\bigrank{big\,rank}
\opn\superheight{superheight}\opn\lcm{lcm}
\opn\trdeg{tr\,deg}%\emph{
\opn\reg{reg} \opn\lreg{lreg} \opn\ini{in} \opn\lpd{lpd}
\opn\size{size}\opn\bigsize{bigsize}
\opn\cosize{cosize}\opn\bigcosize{bigcosize}
\opn\sdepth{sdepth}\opn\sreg{sreg}
\opn\link{link}\opn\fdepth{fdepth}
\opn\index{index}
\opn\index{index}
\opn\indeg{indeg}
\opn\N{N}
\opn\SSC{SSC}
\opn\SC{SC}
\opn\conv{conv}
\opn\div{div} \opn\Div{Div} \opn\cl{cl} \opn\Cl{Cl}
\opn\Spec{Spec} \opn\Supp{Supp} \opn\supp{supp} \opn\Sing{Sing}
\opn\Ass{Ass} \opn\Min{Min}\opn\Mon{Mon} \opn\dstab{dstab} \opn\astab{astab}
\opn\Syz{Syz}
\opn\reg{reg}
\opn\Ann{Ann} \opn\Rad{Rad} \opn\Soc{Soc}
\opn\Im{Im} \opn\Ker{Ker} \opn\Coker{Coker} \opn\Am{Am}
\opn\Hom{Hom} \opn\Tor{Tor} \opn\Ext{Ext} \opn\End{End}
\opn\Aut{Aut} \opn\id{id}

\opn\nat{nat}
\opn\pff{pf}%   \pf exists already
\opn\Pf{Pf} \opn\GL{GL} \opn\SL{SL} \opn\mod{mod} \opn\ord{ord}
\opn\Gin{Gin} \opn\Hilb{Hilb}\opn\sort{sort}
\opn\initial{init}
\opn\ende{end}
\opn\height{height}
\opn\type{type}
\opn\aff{aff} \opn\con{conv} \opn\relint{relint} \opn\st{st}
\opn\lk{lk} \opn\cn{cn} \opn\core{core} \opn\vol{vol}
\opn\link{link} \opn\star{star}\opn\lex{lex}\opn\Mon{Mon}\opn\Min{Min}
\opn\gr{gr}

\def\pot#1#2{#1[\kern-0.28ex[#2]\kern-0.28ex]}

\opn\dirlim{\underrightarrow{\lim}}
\opn\inivlim{\underleftarrow{\lim}}
\def\Implies{\ifmmode\Longrightarrow \else
        \unskip${}\Longrightarrow{}$\ignorespaces\fi}
\def\implies{\ifmmode\Rightarrow \else
        \unskip${}\Rightarrow{}$\ignorespaces\fi}
\def\iff{\ifmmode\Longleftrightarrow \else
        \unskip${}\Longleftrightarrow{}$\ignorespaces\fi}

\let\:=\colon
\newtheorem{Theorem}{Theorem}[section]
 \newtheorem{Lemma}[Theorem]{Lemma}

 \newtheorem{Example}[Theorem]{Example}

 \newtheorem*{Definition*}{Definition}

 \newtheorem*{Conjecture*}{Conjecture}

\let\epsilon\varepsilon
\let\kappa=\varkappa
\def\qed{\ifhmode\textqed\fi
      \ifmmode\ifinner\quad\qedsymbol\else\dispqed\fi\fi}
\def\textqed{\unskip\nobreak\penalty50
       \hskip2em\hbox{}\nobreak\hfil\qedsymbol
       \parfillskip=0pt \finalhyphendemerits=0}
\def\dispqed{\rlap{\qquad\qedsymbol}}

\opn\dis{dis}
\def\pnt{{\raise0.5mm\hbox{\large\bf.}}}

\opn\Lex{Lex}

\begin{document}

 \title{Koszul Filtrations and finite lattices}

 \author {Dancheng Lu}
 \author{Ke Zhang}

 \begin{abstract}
In this note, we  characterize when a finite lattice is distributive  in terms of the existences of some particular classes of   Koszul filtrations.
 \end{abstract}

\address{Dancheng Lu, Department of Mathematics, Soochow University, P.R.China} \email{ludancheng@suda.edu.cn}
\address{Ke Zhang, Department of Mathematics, Soochow University, P.R.China} \email{319793057@qq.com}

 \maketitle

\section{Introduction}

Let $K$ be a field and $L$ a finite lattice. We use $K[L]$ to denote the polynomial ring over a field $K$ whose variables are elements of $L$. A binomial in $K[L]$ of the form $ab-(a\vee b)(a\wedge b)$, where $a,b\in L$ are incomparable, is called a basic binomial or a Hibi relation. By definition the {\it join-meet ideal} $I_L$  is the ideal of $K[L]$ generated by all basic  binomials.  Set $H[L]:=K[L]/I_L$.  It was shown in \cite{H} that  $H[L]$ is a toric ring if and only if $L$ is distributive. In this case, $H[L]$ is called a {\it Hibi ring}.

   Assume that $R$ is a standard graded $K$-algebra. Recall that a collection  $\mathcal{F}$ of ideals of $R$ is  a {\it Koszul filtration} if

(1) Every ideal in $\mathcal{F}$ is generated by linear forms,

(2) The ideals $0$ and the maximal graded ideal $\mathrm{m}$ of $R$ belong to $\mathcal{F}$,

(3) For any ideal $0\neq I\in \mathcal{F}$, there exists an ideal $J\in \mathcal{F}$ such that $J\subset I$, $I/J$ is cyclic and $J:I\in \mathcal{F}$.

This notion, firstly introduced in \cite{CTV}, was inspired by the work of  Herzog, Hibi and Restuccia \cite{HHR} on strongly
Koszul algebras. Its significance  is that if $R$ admits a Koszul filtration then $R$ is Koszul, that is, the residue field $K$ has a linear $R$-free resolution as an $R$-module, thus it provides an  effective way to show  a standard graded algebra to be Koszul.

\vspace{1mm}
   The $K$-algebra $H[L]$ is  standard graded by setting $\mathrm{deg}(a)=1$ for each $a\in L$. Recall that a subset $J$ of a lattice $L$ is called a {\it poset ideal} if for any $a,b\in L$ with $a\leq b$ and $b\in J$ one has $a\in J$. Let $J$ be a poset ideal of $L$. We denote by $(\overline{J})$ the ideal of $H[L]$ generated by elements $\overline{a}=:a+I_L$ with  $a\in J$.
The ideal of the form $(\overline{J})$ is called a {\it poset ideal} of $H[L]$.
It was  proved in \cite{EHH} that if $L$ is  distributive then all the poset ideals of $H[L]$ form a Koszul filtration of $H[L]$.

\vspace{1mm}
The objective of this note is to characterize when a finite lattice is distributive by the existences of some particular classes of Koszul filtrations. We firstly  show that $L$  is distributive if $H[L]$ admits a Koszul filtration in which every element is a poset ideal. Next we introduce the notion of a combinatorial Koszul filtration. By definition, a Koszul filtration $\mathcal{F}$ of $H[L]$ is  {\it combinatorial} if every ideal of $\mathcal{F}$  is generated by residue classes of some elements of $L$ (i.e., variables). We show that  a modular lattice $L$ is distributive if and only if $H[L]$ admits a combinatorial Koszul filtration. An example (Example~\ref{2.2}) is given to show that the  restrictive attribute ``modular" cannot be removed from the last statement.

\begin{figure}[ht!]

\begin{tikzpicture}[line cap=round,line join=round,>=triangle 45,x=1.5cm,y=1.5cm]

\draw (7,3)-- (6,2)--(7,1);
\draw (7,3)-- (7,2)--(7,1);
\draw (7,3)-- (8,2)--(7,1);
\draw (7,3) node[anchor=south east]{$f$};

\draw (7.1,0.7) node[anchor=south east]{$e$};

\draw (5.9,1.8) node[anchor=south east]{$x$};
\draw (6.9,1.8) node[anchor=south east]{$y$};
\draw (8.3,1.8) node[anchor=south east]{$z$};
\draw (8,0.5) node[anchor=north east]{Diamond Lattice};

\fill [color=black] (7,3) circle (1.5pt);
\fill [color=black] (7,1) circle (1.5pt);\fill [color=black] (6,2) circle (1.5pt);\fill [color=black] (8,2) circle (1.5pt);
\fill [color=black] (7,2) circle (1.5pt);

\draw (11,3)--(12,2);
\draw (12,2)--(11,1);
\draw (12.3,0.5) node[anchor=north east]{Pentagon Lattice};

\fill [color=black] (11,3) circle (1.5pt);
\fill [color=black] (10,2.4) circle (1.5pt);
\fill [color=black] (10,1.6) circle (1.5pt);\fill [color=black] (11,1) circle (1.5pt);\fill [color=black] (12,2) circle (1.5pt);
\draw (11,3)--(10,2.4)--(10,1.6)--(11,1);
\draw (11.1,3) node[anchor=south east]{$f$};

\draw (9.9,2.2) node[anchor=south east]{$x$};
\draw (9.9,1.4) node[anchor=south east]{$y$};
\draw (12.25,1.8) node[anchor=south east]{$z$};
\draw (11.2,0.7) node[anchor=south east]{$e$};
\end{tikzpicture}
\caption{}\label{G}
\end{figure}

\section{Characterizations of distributive lattices}

We refer readers to \cite{S} for basic knowledges on finite lattices.  For a finite lattice $L$, we denote by $\max L$ and $\min L$ its largest and least elements  respectively. A finite lattice $L$ is called {\it modular} if $x\leq b$ implies $x\vee (a\wedge b)=(x\vee a)\wedge  b$ for all $x,a,b\in  L$. A finite lattice is modular if and only if no sublattice of $L$ is isomorphic to the pentagon lattice of Figure 1. Here a nonempty subset $L'$ of $L$ is called a {\it sublattice} of $L$ if for any $a,b\in L'$,  both $a\vee b$ and $a\wedge b$ belong to $L'$. A finite lattice $L$ is called {\it distributive} if, for all $x,y,z\in L$, the distributive laws $x\wedge(y \vee  z)=( x\wedge y)\vee(y \wedge z)$ and $x\vee(y \wedge z)=( x\vee y)\wedge(y \vee z)$ hold. Every distributive lattice is modular. A modular lattice is distributive if and only if no sublattice of $L$ is isomorphic to the diamond lattice in Figure~\ref{G}.

\begin{Theorem} Let $L$ be a finite lattice and $R=H[L]$. The following statements are equivalent:

$\mathrm{(1)}$ $L$ is distributive;

$\mathrm{(2)} $ All poset ideals of $R$ form a Koszul filtration of  $R$;

$\mathrm{(3) }$ $R$ admits  a Koszul filtration in which every ideal is a  poset ideal.

\end{Theorem}

\begin{proof} (1)$\Rightarrow$ (2) It follows from \cite[Corollary 2.6]{EHH}.

(2)$\Rightarrow$ (3) Trivially.

(3)$\Rightarrow$ (1)  Suppose that $L$ is not distributive. Then $L$ admits a sublattice $L'$ which is either a diamond lattice or a pentagon lattice as in Figure~\ref{G}. Set $e=\min L'$ and $f=\max L'$.  We first prove the following claim.

\vspace{1mm}
{\it Claim:} If $I$ is a poset ideal of $L$ and   $e$ is a maximal element of $I$, then $(\overline{J}):\overline{e}$ is not generated by linear forms. Here $J:=I\setminus \{e\}$.

\vspace{1mm}
In order to prove our claim, we first show that $[(I_L,J):e]_1$, the  linear part of the colon ideal $(I_L,J):e$, is the $K$-span  of $H$ in $K[L]$, where $H$ is the poset ideal $\{a\in L\:\; a\ngeq e\}$. Let $a\in H$. If $a<e$  then $a\in J$; if $a,e$ are incomparable then $ae=ae- (a\wedge e)(a\vee e)+(a\wedge e)(a\vee e)\in (I_L,J)$. Hence $H \subseteq [(I_L,J):e]_1$. For the converse, let $d\in [(I_L,J):e]_1$. Note that $(I_L,J)=A_1+A_2+A_3$, where $A_1=(J,eH)$, $A_2$ is the ideal of $H[L]$ generated by binomials $ab-(a\vee b)(a\wedge b)$ with $e\notin \{a,b,a\vee b,a\wedge b\}$ and $A_3$ is the ideal of $H[L]$ generated by binomials $ab-e(a\vee b)$ with $a\wedge b=e$. Since $ed\in (I_L,J)$, there is a decomposition $ed=d_1+d_2+d_3$ such that $d_i\in A_i$ and $\deg (d_i)=2$  for $i=1,2,3$. For each $i$, write $d_i$ uniquely  as $d_i=eg_i+h_i$, such that every monomial in the support of  $h_i$ is not divided by $e$. Since $A_1$ is a monomial ideal, $eg_1\in A_1$ and $h_1\in A_1$, and this implies $g_1,h_1\in (J,H)=(H)$. Since $eg_2+h_2\in A_2$, there exist a positive integer $\ell$ and  $k_1,\ldots,k_{\ell}\in K$ such that $$eg_2+h_2=k_1(a_1b_1-(a_1\wedge b_1)(a_1\vee b_1))+\cdots+k_{\ell}(a_{\ell}b_{\ell}-(a_{\ell}\wedge b_{\ell})(a_{\ell}\vee b_{\ell})).   \qquad (1)$$ Here $e\notin\{a_j, b_j, a_j\wedge  b_j, a_j\vee b_j\}$ for $j=1,\ldots,\ell$. It follows from (1) that $g_2=0$ and $h_2\in A_2$. Since $eg_3+h_3\in A_3$, there exist a positive integer $l$ and $k'_1,\ldots,k'_l\in K$ such that $$eg_3+h_3=k'_1(a_1b_1-e(a_1\vee b_1))+\cdots+k'_l(a_lb_l-e(a_l\vee b_l)).   \qquad (2)$$ Here $a_j\wedge b_j=e$ for $j=1,\ldots,l$. It follows from (2) that $h_3=k'_1a_1b_1+\cdots+k'_la_lb_l$. Note that $h_1+h_2+h_3=0$ and the support of $h_3$ is disjoint with the support of $h_i$ for $i=1,2$, one has $h_3=0$ and this implies  $k'_j=0$ for $j=1,\ldots,l$. In particular, $g_3=0$. Therefore, $d=g_1$ belongs to the $K$-span of $H$ in $K[L]$.

\vspace{1mm}
 Assume now on the contrary that  $(\overline{J}):\overline{e}$ is  generated by linear forms.  Since $e(fx-fy)\equiv yzx-xzy\equiv 0 \ (\mathrm{mod}\ I_L)$, we obtain $fx-fy \in (I_L,J):e$. Here, $e,f,x,y,z$ are all elements of $L'$,  see Figure~\ref{G}. It follows that $fx-fy\in (I_L,H)$ from the assumption together with the conclusion of the preceding paragraph.  Thus we can express $fx-fy$ as $fx-fy$
  $$=\sum_{i\in A}k_i(a_ib_i-fx)+\sum_{i\in B}k_i(a_ib_i-fy)+\sum_{i\in C}k_i(a_ib_i-(a_i\vee b_i)(a_i\wedge b_i))+\sum_{i\in D}k_ih_it_i.\ (3)$$
Here,  $a_i\vee b_i=f$ and $a_i\wedge b_i=x$ for $i\in A$,  $a_i\vee b_i=f$ and $a_i\wedge b_i=y$ for $i\in B$, $(a_i\vee b_i)(a_i\wedge b_i)\notin\{fx,fy\}$ for $i\in C$ and $h_i\in H,t_i\in L$ for $i\in D$. Note that $h_jt_j\neq a_ib_i$ if $i\in A$ and $h_j\in H$. Comparing the coefficients of $a_ib_i$ and $fx$ in (3) respectively one has $k_i=0$ for each $i\in A\cup B$ and $\sum_{i\in A}k_i=-1$, a contradiction. Thus we  complete the proof of our  claim.

\vspace{1mm}
  If $R$ has  a Koszul filtration $\mathcal{F}$ consisting  of poset ideals, then there is a poset ideal, say $(\overline{I})$, which is minimal among poset ideals of $\mathcal{F}$ containing $\overline{e}$. Then $e$ is a maximal element of $I$ and $J=I\setminus \{e\}$ is a poset ideal of $L$. It follows that $(\overline{J})$ is a unique poset ideal contained in $(\overline{I}$) such that $(\overline{I})/(\overline{J})$ is cyclic and $(\overline{J})\in \mathcal{F}$. However by the claim, $(\overline{J}):\overline{e}=(\overline{J}):(\overline{I})$ is not generated by linear forms, a contradiction.  \end{proof}

  We say that a Koszul filtration $\mathcal{F}$ of $H[L]$ is {\it combinatorial} if every ideal in $\mathcal{F}$  is generated by the residue classes of some elements of $L$ (i.e., variables). It is natural to ask if only for a distributive lattice $L$, the algebra $H[L]$ has a combinatorial Koszul filtration. This is not the case as shown by the following example.

\begin{Example} \label{2.2}  \em{
 Let $P$ be the pentagon lattice as in Figure~\ref{G}. Then $H[P]$ has the following combinatorial Koszul filtration:
$$(0), (\overline{x}), (\overline{x},\overline{y}), (\overline{x},\overline{z}), (\overline{x},\overline{y},\overline{z}), (\overline{x},\overline{y},\overline{z},\overline{e}), (\overline{x},\overline{y},\overline{z},\overline{f}), (\overline{x},\overline{y},\overline{z},\overline{e},\overline{f}).$$  One can check the following equalities by Singular \cite{DGPS}:}

$0:(\overline{x})=0$; $(\overline{x}):(\overline{y})=(\overline{z},\overline{x})$; $(\overline{x}):(\overline{z})=(\overline{y},\overline{x})$;  $(\overline{x},\overline{y}):(\overline{z})=(\overline{x},\overline{y})$; $(\overline{x},\overline{y},\overline{z}):\overline{e}=(\overline{x},\overline{y},\overline{z},\overline{f})$; $(\overline{x},\overline{y},\overline{z}):\overline{f}=(\overline{x},\overline{y},\overline{z},\overline{e})$; $(\overline{x},\overline{y},\overline{z},\overline{e}):\overline{f}=(\overline{x},\overline{y},\overline{z},\overline{e}).$
\end{Example}

We need to introduce some more notation. A finite lattice is called {\it pure} if all maximal chains (totally ordered subsets) have the same length. When a finite lattice is pure,  the rank of $a$ in $L$, denoted by $\mathrm{rank}(a)$, is the largest integer $r$ for which there exists a chain of $L$ of the form $a_0 <a_1<\cdots <a_r = a$. If  $L$ is modular, then $L$ is pure and the following  equality  holds for any $p,q\in L$:$$\mathrm{rank}(p)+\mathrm{rank}(q)=\mathrm{rank}(p\wedge q)+\mathrm{rank}(p\vee q).$$

We record  \cite[lemma 1.2]{EH} in the following lemma for the later use.

\begin{Lemma} \label{lemma} Let $L$  be a modular non-distributive lattice. Then $L$ has a diamond sublattice $L'$ such that $\mathrm{rank}( \max L')-\mathrm{rank} (\min L')=2$.
\end{Lemma}

\begin{Theorem} \label{modular} Let $L$ be a modular lattice. Then $L$ is distributive if and only if $H[L]$ admits a combinatorial Koszul filtration.
\end{Theorem}
\begin{proof} If $L$ is distributive, then $H[L]$ admits a Koszul filtration consisting of poset ideals by \cite[Corollary 2.6]{EHH},  which is certainly  combinatorial. Thus the direction ``only if" is proved.

  Suppose now that $L$ is non-distributive. Then $L$ has a diamond sublattice $L'$ such that $\mathrm{rank}(\max L')-\mathrm{rank}( \min L')=2$ by Lemma~\ref{lemma}.
Set $e=\min L'$ and $f= \max L'$. Then for any distinct $x,y$ in the open interval $(e,f)$, they are incomparable and $x\vee y=f$ and $x\wedge y=e$. For convenience, we write $(e,f)=\{x_1,x_2,\ldots,x_n\}$ for some $n\geq 3$.

\vspace{1mm}
{\it  Claim:} If $S$ is a subset of $L$ such that $S\cap [e,f]=\emptyset$, then $(\overline{S}):\overline{x}$ cannot be  generated by residue classes of some variables for any $x\in [e,f]$.
 \vspace{1mm}

 We first consider the case  when  $x\in (e,f)$, say $x=x_1$.  Let us show that $x_1x_j\notin (I_L,S)$ for $j=2,\ldots,n$. In fact, if $x_1x_j\in (I_L,S)$ for some $j\neq 1$, then there exist $k_i,l_i\in K$ (the ground field) and $t_i\in L$ such that  $$x_1x_j=\sum_{a_i,b_i} k_i(a_ib_i-(a_i\vee b_i)(a_i\wedge b_i))+\sum_{s_i\in S} l_is_it_i. \qquad (4)$$ Here $\{a_i,b_i\}$ ranges through all incomparable pairs of $L$.
   Without loss of generality we assume that $a_1b_1=x_1x_j$ and $\{a_i,b_i\}\subseteq (e,f)$ for $i\leq m(:=n(n-1)/2)$ and $\{a_i,b_i\}\nsubseteq (e,f)$ for $i>m$. This implies $(a_i\vee b_i)(a_i\wedge b_i)=ef$ if $1\leq i\leq m$ and $(a_i\vee b_i)(a_i\wedge b_i)\neq ef$ if $i>m$. Note that $s_jt_j\neq a_ib_i$ if $s_j\in S$ and $i\leq m$. Comparing the coefficients
of $a_ib_i$ with $i\leq m$ and $ef$  in (4) respectively, we obtain $k_1=1$, $k_2=\cdots=k_m=0$ and $k_1+k_2+\cdots+k_m=0$, a contradiction. Thus $x_j\notin (I_L,S):x_1$ for each $j>1$. But one has $(x_2-x_3)$ belongs to  $(I_L,S):x_1$ and this implies that $(\overline{S}):\overline{x_1}$ is not generated by the residue classes of variables.

 \vspace{1mm}

 For the case when $x=e$, we first see that both $e(f+k_1a_1+\cdots+k_ra_r)$ and $e(x_j+k_1b_1+\cdots+k_vb_v)$ do not belong to $(I_L,S)$ for any positive integers $r,v$, any $k_i\in K$, $f\neq a_i\in L$, $x_j\neq b_i\in L$ and $j=1,\ldots,n$. This fact can be proved  in a similar manner as we prove $x_1x_j\notin (I_L,S)$ in the preceding paragraph and so we omit its proof. It follows that neither $f$ nor $x_i,i=1,\ldots,n$ appears in the support of any linear polynomial in $(I_L,S):e$.  Thus, if $(\overline{S}):\overline{e}$ is generated by linear forms, then
  $f(x_1-x_2)$ does not belong to $(I_L,S):e$. This leads to a contradiction, since $ef(x_1-x_2)\in I_L$. Hence $(\overline{S}):\overline{e}$
is not generated by linear forms. The final case when $x=f$ can be proved in the same way as in the case when $x=e$. Thus our claim has been proved.

 \vspace{1mm}
Now assume on the contrary that $H[L]$ admits a combinatorial Koszul filtration $\mathcal{F}$. Note that $\mathcal{F}$ has a natural partial order induced by inclusion. Let $(\overline{T})$ be a minimal element in $\mathcal{F}$ satisfying $T\cap [e,f]\neq \emptyset$ . Then $T\cap [e,f]$ consists of a single  element, say $x$. Moreover $(\overline{T\setminus \{x\}})$ is a unique element in $\mathcal{F}$
such that $(\overline{T})/(\overline{T\setminus \{x\}})$ is cyclic. It follows that  $(\overline{T\setminus \{x\}}):\overline{x}\in \mathcal{F}$,   which is contradicted to our claim.
\end{proof}

\begin{Example} {\em Let $D$ be the diamond lattice as in Figure~\ref{G}. Then $H[D]$ admits no combinatorial Koszul filtrations by Theorem~\ref{modular}. However $H[D]$ has the following Koszul filtration:$$(0),(\overline{x}),(\overline{y}-\overline{z}),(\overline{x},
\overline{y}),(\overline{x},\overline{z}),(\overline{x},\overline{y},\overline{z}),
(\overline{x},\overline{y},\overline{z},\overline{e}),(\overline{x},\overline{y},
\overline{z},\overline{f}),(\overline{x},\overline{y},\overline{z},\overline{e},\overline{f}).$$
One can check the following equalities  by Singular \cite{DGPS}:

\noindent$(0):(\overline{x})=(\overline{y}-\overline{z}); $ $(0):(\overline{y}-\overline{z})=(\overline{x})$; $(\overline{x}):(\overline{y})=(\overline{x},\overline{z});$ $(\overline{x}):(\overline{z})=(\overline{x},\overline{y});$ $(\overline{x},\overline{y}):(\overline{z})=(\overline{x},\overline{y});$
$(\overline{x},\overline{y},\overline{z}):(\overline{e})=(\overline{x},\overline{y},\overline{z},\overline{f});$
$(\overline{x},\overline{y},\overline{z}):(\overline{f})=(\overline{x},\overline{y},\overline{z},\overline{e});$
$(\overline{x},\overline{y},\overline{z},\overline{e}):(\overline{f})=(\overline{x},\overline{y},\overline{z},\overline{e}).$
 }
\end{Example}

It would be of interest to know if $H[L]$ admits a Koszul filtration for any finite lattice $L$.

\vspace{2mm}

\noindent {\bf Acknowledge:}  Thank the referee very much for his/her careful reading and  interesting comments!

\end{document}